\documentclass[12pt,a4paper]{article}
\usepackage{amssymb}
\usepackage{amsmath}
\usepackage{latexsym}
\usepackage{amsthm}

\def\E{\mathbb E}

\newtheorem{lem}{Lemma}
\newtheorem{thm}{Theorem}
\newtheorem{thmletter}{Theorem}

\newtheorem{cor}[thm]{Corollary}
\newtheorem{prop}[thm]{Proposition}
\theoremstyle{definition}
\newtheorem*{remark}{Remark}
\theoremstyle{definition}

\usepackage{xpatch}
\xpatchcmd{\proof}{\itshape}{\normalfont\proofnameformat}{}{}
\newcommand{\proofnameformat}{}

\begin{document}

\renewcommand{\proofnameformat}{\bfseries}

\begin{center}
{\large\textbf{Berry--Esseen smoothing inequality for the Wasserstein metric on compact Lie groups}}

\vspace{5mm}

\textbf{Bence Borda}

{\footnotesize Graz University of Technology

Steyrergasse 30, 8010 Graz, Austria

and

Alfr\'ed R\'enyi Institute of Mathematics

Re\'altanoda utca 13--15, 1053 Budapest, Hungary

Email: \texttt{borda@math.tugraz.at}}

\vspace{5mm}

{\footnotesize \textbf{Keywords:} transport metric, Fourier transform, compact group, random walk, spectral gap, Erd\H{o}s--Tur\'an inequality, equidistribution}

{\footnotesize \textbf{Mathematics Subject Classification (2020):} 43A77, 60B15}
\end{center}

\vspace{5mm}

\begin{abstract}
We prove a sharp general inequality estimating the distance of two probability measures on a compact Lie group in the Wasserstein metric in terms of their Fourier transforms. We use a generalized form of the Wasserstein metric, related by Kantorovich duality to the family of functions with an arbitrarily prescribed modulus of continuity. The proof is based on smoothing with a suitable kernel, and a Fourier decay estimate for continuous functions. As a corollary, we show that the rate of convergence of random walks on semisimple groups in the Wasserstein metric is necessarily almost exponential, even without assuming a spectral gap. Applications to equidistribution and empirical measures are also given.
\end{abstract}

\section{Introduction}

If the Fourier transform of two Borel probability measures on $\mathbb{R}$ are equal, then the measures themselves are also equal. The celebrated Berry--Esseen smoothing inequality is a quantitative form of this fundamental fact of classical Fourier analysis. Given two Borel probability measures $\nu_1$ and $\nu_2$ on $\mathbb{R}$, let $F_j (x) = \nu_j ((-\infty, x])$, $j=1,2$, and define
\[ \delta_{\mathrm{unif}} (\nu_1, \nu_2) = \sup_{x \in \mathbb{R}} \left| F_1 (x)-F_2(x) \right| . \]
\begin{thmletter}[Berry--Esseen smoothing inequality]\label{classicalBerryEsseen} Assume that $|F_2(x)-F_2(y)|\le K |x-y|$ for all $x,y \in \mathbb{R}$ with some constant $K>0$. Then for any real number $T>0$,
\[ \delta_{\mathrm{unif}} (\nu_1, \nu_2) \ll \frac{K}{T} + \int_{-T}^T \frac{|\widehat{\nu_1}(t) -\widehat{\nu_2}(t)|}{|t|} \, \mathrm{d}t \]
with a universal implied constant.
\end{thmletter}
\noindent In the terminology of probability theory, $F_j (x)$ is the distribution function of $\nu_j$; the Fourier transform $\widehat{\nu_j}(t)= \int_{\mathbb{R}}e^{itx}\, \mathrm{d}\nu_j (x)$ is the characteristic function of $\nu_j$; finally, $\delta_{\mathrm{unif}}$ is the uniform metric (or Kolmogorov metric) on the set of probability distributions. For somewhat sharper forms of Theorem \ref{classicalBerryEsseen} see Petrov \cite[Chapter 5.1]{Petrov}.

Similar smoothing inequalities are known for several other probability metrics on $\mathbb{R}$, see Bobkov \cite{Bobkov} for a survey. Some, but not all require a smoothness assumption on one of the distributions; for instance, Theorem \ref{classicalBerryEsseen} is usually formulated under the assumption that $F_2$ is differentiable and $|F_2'(x)|\le K$. A common feature of such results is that the distance of $\nu_1$ and $\nu_2$ in some probability metric is bounded above by the sum of two terms depending on a free parameter $T>0$: one term decays as $T$ increases, and the other term depends on the Fourier transforms of $\nu_1$ and $\nu_2$ only on the interval $[-T,T]$.

Berry--Esseen type smoothing inequalities are known in other spaces as well. The first multidimensional version, an upper bound for the uniform metric on $\mathbb{R}^d$ is due to von Bahr \cite{vonBahr}. Niederreiter and Philipp proved an analogous result for two Borel probability measures $\nu_1$ and $\nu_2$ on the torus $\mathbb{R}^d/\mathbb{Z}^d$. By identifying $\mathbb{R}^d/\mathbb{Z}^d$ with the unit cube $[0,1)^d$, we can define the uniform metric as $\delta_{\mathrm{unif}} (\nu_1, \nu_2) = \sup_{x \in [0,1]^d} |\nu_1 (B(x)) - \nu_2(B(x))|$, where $B(x)=[0,x_1)\times \cdots \times [0,x_d)$. The Fourier transform is now $\widehat{\nu_j}(m)=\int_{\mathbb{R}^d/\mathbb{Z}^d} e^{-2\pi i \langle m,x \rangle} \, \mathrm{d} \nu_j(x)$, $m \in \mathbb{Z}^d$. Let $\mu_{\mathbb{R}^d/\mathbb{Z}^d}$ be the normalized Haar measure, and let $\| m \|_{\infty}=\max_{1 \le k \le d} |m_k|$.
\begin{thmletter}[Niederreiter--Philipp \cite{Niederreiter}]\label{NiederreiterBerryEsseen} Assume that $\nu_2(B) \le K \mu_{\mathbb{R}^d/\mathbb{Z}^d}(B)$ for all axis parallel boxes $B \subseteq [0,1)^d$ with some constant $K>0$. Then for any real number $M>0$,
\[ \delta_{\mathrm{unif}} (\nu_1, \nu_2) \ll \frac{K}{M} + \sum_{\substack{m \in \mathbb{Z}^d \\ 0<\| m \|_{\infty}<M}} \frac{|\widehat{\nu_1}(m)-\widehat{\nu_2}(m)|}{\prod_{k=1}^d \max \{|m_k|, 1\}} \]
with an implied constant depending only on $d$.
\end{thmletter}

The goal of this paper is to prove a Berry--Esseen type smoothing inequality in more general compact groups. In this more general setting only those probability metrics remain meaningful whose definition does not rely on concepts such as axis parallel boxes and distribution functions. One of the most important such metrics is the $p$-Wasserstein metric $W_p$. Given a compact metric space $(X,\rho)$ and two Borel probability measures $\nu_1$ and $\nu_2$ on $X$, we define
\[ W_p (\nu_1, \nu_2) = \inf_{\vartheta \in \mathrm{Coup}(\nu_1, \nu_2)} \int_{X \times X} \rho (x,y)^p \, \mathrm{d} \vartheta (x,y) \qquad (0<p \le 1), \]
and
\[ W_p (\nu_1, \nu_2) = \inf_{\vartheta \in \mathrm{Coup}(\nu_1, \nu_2)} \left( \int_{X \times X} \rho (x,y)^p \, \mathrm{d} \vartheta (x,y) \right)^{1/p} \qquad (1<p<\infty). \]
Here $\mathrm{Coup}(\nu_1, \nu_2)$ is the set of couplings of $\nu_1$ and $\nu_2$; that is, the set of Borel probability measures $\vartheta$ on $X \times X$ with marginals $\vartheta (B \times X)=\nu_1 (B)$ and $\vartheta (X \times B) = \nu_2 (B)$, $B \subseteq X$ Borel. Recall that for any $p>0$, $W_p$ is a metric on the set of Borel probability measures on $X$, and it generates the topology of weak convergence. Observe also the general inequalities $W_p(\nu_1, \nu_2) \le W_1 (\nu_1, \nu_2)^p$, $0<p \le 1$ and $W_1(\nu_1, \nu_2) \le W_p (\nu_1, \nu_2)$, $1 \le p < \infty$.

Respecting the philosophy of the Berry--Esseen inequality, we wish to find an upper bound to $W_p (\nu_1, \nu_2)$ depending on the Fourier transform of $\nu_1$ and $\nu_2$ only up to a certain ``level''. For this reason we chose to work with compact Lie groups, where the theory of highest weights provides a suitable framework to formalize the meaning of ``level''. More precisely, our main result applies to any compact, connected Lie group $G$; classical examples include $\mathbb{R}^d/\mathbb{Z}^d$, $\mathrm{U}(d)$, $\mathrm{SU}(d)$, $\mathrm{SO}(d)$, $\mathrm{Sp}(d)$ and $\mathrm{Spin} (d)$. Let $\widehat{G}$ denote the unitary dual, and let $d_{\pi}$, $\lambda_{\pi}$ and $\kappa_{\pi}$ denote the dimension, the highest weight and the negative Laplace eigenvalue of the representation $\pi \in \widehat{G}$, respectively. Further, let $\| A \|_{\mathrm{HS}}=\sqrt{\mathrm{tr} (A^*A)}$ be the Hilbert--Schmidt norm of a matrix $A$. For a more formal setup we refer to Section \ref{notation}. Generalizing our recent result on the torus $G=\mathbb{R}^d/\mathbb{Z}^d$ \cite{Borda}, in this paper we prove the following Berry--Esseen type smoothing inequality for $W_p$ on compact Lie groups.
\begin{thm}\label{WpBerryEsseen} Let $\nu_1$ and $\nu_2$ be Borel probability measures on a compact, connected Lie group $G$. For any $0<p \le 1$ and any real number $M>0$,
\begin{equation}\label{Wp}
W_p (\nu_1, \nu_2) \ll \frac{1}{M^p} + M^{1-p} \left( \sum_{\substack{\pi \in \widehat{G} \\ 0< |\lambda_{\pi}|<M}} \frac{d_{\pi}}{\kappa_{\pi}} \| \widehat{\nu_1}(\pi) - \widehat{\nu_2}(\pi) \|_{\mathrm{HS}}^2 \right)^{1/2}
\end{equation}
with an implied constant depending only on $G$.
\end{thm}
\noindent The result holds without any smoothness assumption on $\nu_1$ and $\nu_2$. As the applications in Section \ref{applications} will show, the inequality is sharp up to a constant factor depending on $G$. Our methods do not work when $p>1$; the reason is that the proof is based on Kantorovich duality for $W_p$. Recall that the Kantorovich duality theorem states that for any $0<p\le 1$,
\[ W_p (\nu_1, \nu_2 ) = \sup_{f \in \mathcal{F}_p} \left| \int_G f \, \mathrm{d}\nu_1 - \int_G f \, \mathrm{d}\nu_2 \right| , \]
where, with $\rho$ denoting the geodesic distance on $G$,
\[ \mathcal{F}_p = \left\{ f:G \to \mathbb{R} \, : \, |f(x)-f(y)| \le \rho (x,y)^p \textrm{ for all } x,y \in G \right\} . \]
Theorem \ref{WpBerryEsseen} thus estimates the difference of the integrals of $f$ with respect to $\nu_1$ and $\nu_2$ uniformly in $f \in \mathcal{F}_p$. From our methods it also follows that for any $f \in \mathcal{F}_p$,
\[ \left| \int_G f \, \mathrm{d}\nu_1 - \int_G f \, \mathrm{d}\nu_2 \right| \ll \frac{1}{M^p} + \sum_{\substack{\pi \in \widehat{G}\\ 0<|\lambda_{\pi}|<M}} d_{\pi} \| \widehat{f}(\pi ) \|_{\mathrm{HS}} \cdot \| \widehat{\nu_1}(\pi) - \widehat{\nu_2}(\pi) \|_{\mathrm{HS}} \]
with an implied constant depending only on $G$; see Proposition \ref{smoothingprop}. Hence for a given $f \in \mathcal{F}_p$ whose Fourier transform decays fast enough, the results of Theorem \ref{WpBerryEsseen} can be improved. Fast Fourier decay follows e.g.\ from suitable smoothness assumptions on $f$, see Sugiura \cite{Sugiura}. We mention that by prescribing a higher order modulus of continuity for $f$, the term $1/M^p$ can also be improved. Note that Fourier decay rates play a role in classical Berry--Esseen type inequalities as well: the coefficient $|t|^{-1}$ (resp.\ $\prod_{k=1}^d \max\{ |m_k|,1 \}^{-1}$) in Theorem \ref{classicalBerryEsseen} (resp.\ Theorem \ref{NiederreiterBerryEsseen}) is explained by the fact that the Fourier transform of the indicator function of an interval (resp.\ axis parallel box) decays at this rate.

The most straightforward application of Theorem \ref{WpBerryEsseen} is estimating the rate of convergence of random walks in the $W_p$ metric. Let $\nu^{*k}$ denote the $k$-fold convolution power of $\nu$, and let $\mu_G$ be the Haar measure on $G$. Recall that $\nu^{*k} \to \mu_G$ weakly as $k \to \infty$ if and only if the support of $\nu$ is contained neither in a proper closed subgroup, nor in a coset of a proper closed normal subgroup of $G$; see Stromberg \cite{Stromberg}. Using a nonuniform spectral gap result of Varj\'u, we prove the following application of Theorem \ref{WpBerryEsseen}.
\begin{cor}\label{convergencerate} Let $\nu$ be a Borel probability measure on a compact, connected, semisimple Lie group $G$. If $\nu^{*k} \to \mu_G$ weakly as $k \to \infty$, then
\[ W_1 (\nu^{*k}, \mu_G ) \ll e^{-ck^{1/3}}, \]
where the constant $c>0$ and the implied constant depend only on $G$ and $\nu$.
\end{cor}
\noindent The condition of semisimplicity cannot be removed. The rate of convergence $W_p \ll \exp (-pck^{1/3})$, $0<p \le 1$ immediately follows. The main motivation came from our recent paper \cite{Borda} on quantitative ergodic theorems for random walks. Given independent, identically distributed $G$-valued random variables $\zeta_1, \zeta_2, \dots$ with distribution $\nu$, we showed that for any $f \in \mathcal{F}_p$ the sum $\sum_{k=1}^N f(\zeta_1 \zeta_2 \cdots \zeta_k)$ satisfies the central limit theorem and the law of the iterated logarithm, provided $\sum_{k=1}^{\infty} W_p (\nu^{*k},\mu_G) < \infty$. Corollary \ref{convergencerate} thus provides a large class of examples of random walks with fast enough convergence in $W_p$, and consequently to which our quantitative ergodic theorems apply. We do not know whether $W_p (\nu^{*k}, \mu_G) \ll e^{-ck^{1/3}}$ remains true for $p>1$.

Another possible application is in uniform distribution theory, where the goal is finding finite sets $\{ a_1, a_2, \dots, a_N \} \subset G$ which make the integration error $|N^{-1} \sum_{k=1}^N f(a_k)-\int_G f \, \mathrm{d}\mu_G|$ small for a suitable class of test functions. Applying Theorem \ref{WpBerryEsseen} to $\nu_1=N^{-1} \sum_{k=1}^N \delta_{a_k}$ (where $\delta_a$ is the Dirac measure concentrated at $a \in G$) and $\nu_2=\mu_G$, we can quantitatively measure how well distributed a finite set is with respect to test functions $f \in \mathcal{F}_p$. Note that in this case we have
\begin{equation}\label{charactersums}
\| \widehat{\nu_1} (\pi ) - \widehat{\nu_2} (\pi ) \|_{\mathrm{HS}}^2 = \frac{1}{N^2} \sum_{k, \ell =1}^N \chi_{\pi} (a_k^{-1}a_{\ell}) ,
\end{equation}
where $\chi_{\pi} (x)=\mathrm{tr} \, \pi (x)$ is the character of the representation $\pi \in \widehat{G}$. Theorem \ref{WpBerryEsseen} thus becomes an abstract version of the Erd\H{o}s--Tur\'an inequality, estimating the distance of a finite set from uniformity in terms of character sums. As an illustration, we will show that certain finite sets in $\mathrm{SO}(3)$ constructed by Lubotzky, Phillips and Sarnak using deep number theory are optimal with respect to $W_p$, $0<p \le 1$.

The only results in the $p>1$ case we are aware of are due to Brown and Steinerberger \cite{Brown} and Steinerberger \cite{Steinerberger1}, who estimated the distance in $W_2$ from $N^{-1}\sum_{k=1}^N \delta_{a_k}$ to the Haar measure in terms of character sums \eqref{charactersums}, and also in terms of the Green function of the Laplace--Beltrami operator; in fact, their results generalize to compact Riemannian manifolds with the Haar measure replaced by the Riemannian volume. A similar Erd\H{o}s--Tur\'an inequality on compact Riemannian manifolds with respect to sufficiently nice test sets was proved by Colzani, Gigante and Travaglini \cite{Colzani}. Numerical results for certain finite point sets on the orthogonal group $\mathrm{O}(d)$ and on Grassmannian manifolds were obtained by Pausinger \cite{Pausinger}. See also Steinerberger \cite{Steinerberger2} for the connection between $W_p$, $p \ge 1$ and Fourier analysis on the torus.

For probability distributions on $\mathbb{R}$, Esseen \cite{Esseen}, \cite[Corollary 8.3]{Bobkov} used a smoothing inequality for $W_1(\nu_1, \nu_2)$ to estimate the rate of convergence in $W_1$ in the central limit theorem; our Theorem \ref{WpBerryEsseen} and Corollary \ref{convergencerate} are far reaching analogues of these classical results in the compact setting. In the $p>1$ case the only known smoothing inequality on $\mathbb{R}$ applies to $W_2 (\nu_1, \nu_2)$ with $\nu_2$ a Gaussian distribution \cite[Theorem 11.1]{Bobkov}.

The discussion above can be generalized from $\mathcal{F}_p$ to the class of functions with an arbitrarily prescribed modulus of continuity, and we will actually work out the details in this generality. In particular, our results apply to any given $f \in C(G)$. The formal setup and notation are given in Section \ref{notation}; we state the general form of Theorem \ref{WpBerryEsseen} with explicit constants in Section \ref{BerryEsseensection}; applications to random walks and to uniform distribution theory are discussed in more detail, and the proof of Corollary \ref{convergencerate} is given in Section \ref{applications}. The proof of the main result, Theorem \ref{WgBerryEsseen} is given in Section \ref{proofs}.

\section{Results}\label{results}

\subsection{Notation}\label{notation}

Throughout the paper $G$ denotes a compact, connected Lie group with identity element $e \in G$ and Lie algebra $\mathfrak{g}$. Let $\exp: \mathfrak{g} \to G$ and $n=\mathrm{dim} \, G$ denote the exponential map and the dimension of $G$ as a real smooth manifold. Fix an Ad-invariant inner product $(\cdot, \cdot)$ on $\mathfrak{g}$, and let $|X|=\sqrt{(X,X)}$, $X \in \mathfrak{g}$. This inner product defines a Riemannian metric on $G$; let $\rho$ denote the corresponding geodesic metric on $G$. The Laplace--Beltrami operator on $G$ is $\Delta = \sum_{k=1}^n X_k X_k$ (as an element of the universal enveloping algebra of $\mathfrak{g}$), where $X_1, \dots, X_n$ is an orthonormal base in $\mathfrak{g}$ with respect to $(\cdot, \cdot)$; this does not depend on the choice of the orthonormal base.

Let $r$ denote the rank of $G$, and fix a maximal torus $T$ in $G$ with Lie algebra $\mathfrak{t}$. Let $\mathfrak{t}^*=\mathrm{Hom}(\mathfrak{t},\mathbb{R})$ denote the dual vector space. The sets
\[ \begin{split} \Gamma &= \left\{ X \in \mathfrak{t} \, : \, \exp (2 \pi X)=e  \right\}, \\ \Gamma^* &= \left\{ \lambda \in \mathfrak{t}^* \, : \, \lambda (X) \in \mathbb{Z} \textrm{ for all } X \in \Gamma \right\} \end{split} \]
are dual lattices of full rank in $\mathfrak{t}$ and $\mathfrak{t}^*$, respectively. The inner product on $\mathfrak{g}$ naturally defines an inner product on $\mathfrak{t}^*$, which we also denote by $(\cdot, \cdot)$; we also write $|\lambda|=\sqrt{(\lambda, \lambda)}$, $\lambda \in \mathfrak{t}^*$. The inner product defines a normalized Lebesgue measure $m$ on $\mathfrak{t}$.

The weights will be considered elements of $\Gamma^*$; the character of $T$ corresponding to $\lambda \in \Gamma^*$ is $\exp (2 \pi X) \mapsto e^{2 \pi i \lambda (X)}$, $X \in \mathfrak{t}$. Let $R$ be the set of roots, and choose a set of positive roots $R^+$; we have $|R|=n-r$ and $|R^+|=(n-r)/2$. Let $\Gamma_{\mathrm{root}}^* \subseteq \Gamma^*$ be the lattice spanned by the roots. Further, let
\[ C^+=\left\{ \lambda \in \mathfrak{t}^* \, : \, (\lambda, \alpha) \ge 0 \textrm{ for all } \alpha \in R^+  \right\} \]
be the dominant Weyl chamber; the set of dominant weights is thus $\Gamma^* \cap C^+$. The Weyl group of $G$ with respect to $T$ will be denoted by $W(G,T)=N_G(T)/T$.

Let $\widehat{G}$ be the unitary dual of $G$. For any $\pi \in \widehat{G}$, let $d_{\pi}$ and $\lambda_{\pi}$ denote the dimension and the highest weight of $\pi$. The map $\pi \mapsto \lambda_{\pi}$ is a bijection from $\widehat{G}$ to the set of dominant weights $\Gamma^* \cap C^+$. Let $\kappa_{\pi} \ge 0$ denote the negative Laplace eigenvalue of $\pi$; that is, $\Delta \pi = -\kappa_{\pi} \pi$ where $\Delta$ acts entrywise. Recall that
\[ \kappa_{\pi} = |\lambda_{\pi}|^2 + 2 (\lambda_{\pi}, \rho^+) \qquad \textrm{and} \qquad d_{\pi} = \frac{\prod_{\alpha \in R^+} (\lambda_{\pi}+ \rho^+ , \alpha )}{\prod_{\alpha \in R^+}(\rho^+, \alpha)}, \]
where $\rho^+ = \sum_{\alpha \in R^+} \alpha /2$ is the half-sum of positive roots; in particular,
\[ |\lambda_{\pi}|^2 \le \kappa_{\pi} \le |\lambda_{\pi}|^2 + O(|\lambda_{\pi}|) \qquad \textrm{and} \qquad d_{\pi} \ll |\lambda_{\pi}|^{(n-r)/2}. \]

Let $\mu_G$ (resp.\ $\mu_T$) denote the normalized Haar measure on $G$ (resp.\ $T$). The Fourier transform of a function $f:G \to \mathbb{C}$ is $\widehat{f}(\pi) = \int_G f(x) \pi(x)^* \, \mathrm{d}\mu_G(x)$, $\pi \in \widehat{G}$; that of a Borel probability measure $\nu$ on $G$ is $\widehat{\nu}(\pi)=\int_G \pi(x)^* \, \mathrm{d}\nu(x)$, $\pi \in \widehat{G}$. Here $\pi(x)^*$ denotes the adjoint of $\pi(x)$, and the integrals are taken entrywise.

Let $g:[0,\infty ) \to [0,\infty)$ be a nondecreasing and subadditive\footnote{That is, $g(t+u) \le g(t)+g(u)$ for all $t,u \ge 0$.} function such that $\lim_{t \to 0^+}g(t)=0$, and define
\[ W_g (\nu_1, \nu_2) = \inf_{\vartheta \in \mathrm{Coup}(\nu_1, \nu_2)} \int_{G \times G} g(\rho (x,y)) \, \mathrm{d} \vartheta (x,y), \]
where $\mathrm{Coup}(\nu_1, \nu_2)$ is the set of couplings, as before. Letting
\[ \mathcal{F}_g = \left\{ f:G \to \mathbb{R} \, : \, |f(x)-f(y)| \le g(\rho (x,y)) \textrm{ for all } x,y \in G \right\} , \]
the Kantorovich duality theorem states
\[ W_g (\nu_1, \nu_2) = \sup_{f \in \mathcal{F}_g} \left| \int_G f \, \mathrm{d} \nu_1 - \int_G f \, \mathrm{d} \nu_2 \right| . \]
Note that $W_g$ is a metric on the set of Borel probability measures on $G$ and it generates the topology of weak convergence, unless $g$ is constant zero. In the special case $g(t)=t^p$, $0<p \le 1$ we write $W_p$ (resp.\ $\mathcal{F}_p$) instead of $W_g$ (resp.\ $\mathcal{F}_g$). We mention that given $f \in C(G)$, the function
\[ g_f(t)=\sup \{ |f(x)-f(y)| \, : \, x,y \in G, \, \rho (x,y) \le t \} \]
is nondecreasing and subadditive, and $\lim_{t \to 0^+} g_f(t)=0$; in fact, $g_f$ is the smallest $g$ for which $f \in \mathcal{F}_g$.
\begin{remark} Kantorovich duality is usually stated for $g(t)=t$, i.e.\ for Lipschitz functions. To see the general case, note that $g(\rho (x,y))$ is another metric on $G$ generating the topology of $G$, unless $g$ is constant zero; the subadditivity of $g$ is needed for the triangle inequality. Kantorovich duality for Lipschitz functions in the $g(\rho (x,y))$ metric thus implies Kantorovich duality for $W_g$ as claimed. Further, since the usual $1$-Wasserstein metric with respect to $g(\rho (x,y))$ generates the topology of weak convergence, so does $W_g$.
\end{remark}

\subsection{Berry--Esseen inequality on compact Lie groups}\label{BerryEsseensection}

Let $\eta : \mathfrak{t} \to \mathbb{R}$ be a $W(G,T)$-invariant smooth function such that $\eta (X) = 0$ for all $|X| \ge 1$, and $0 \le \eta (X) \le \eta (0)=1$ for all $X \in \mathfrak{t}$. Since $W(G,T)$ acts by orthogonal transformations on $\mathfrak{t}$, $W(G,T)$-invariance can be ensured e.g.\ if $\eta (X)$ depends only on $|X|$. For instance, we can use the ``bump function''
\[ \eta (X) = \left\{ \begin{array}{ll} \exp \left( - \frac{|X|^2}{1-|X|^2} \right) & \textrm{if } |X| <1, \\ 0 & \textrm{if } |X| \ge 1 . \end{array} \right. \]
Let $F: \mathfrak{t} \to \mathbb{C}$, $F(X) = \int_{\mathfrak{t}} \eta (Y) e^{2 \pi i (X,Y)} \, \mathrm{d}m(Y)$; note that $F(X)$ is a Schwarz function, thus $|F(X)|$ decays at an arbitrary polynomial rate as $|X| \to \infty$. The main result of the paper is the following Berry--Esseen type inequality.
\begin{thm}\label{WgBerryEsseen} Let $\nu_1$ and $\nu_2$ be Borel probability measures on a compact, connected Lie group $G$, and let $g: [0,\infty ) \to [0,\infty )$ be nondecreasing and subadditive such that $\lim_{t \to 0^+}g(t)=0$. Let
\[ \psi (t)= \frac{2}{|W(G,T)|} \int_{\mathfrak{t}} g \left( \frac{2 \pi |X|}{\lfloor t/(|2\rho^+ |+a) \rfloor} \right) a^r |F(aX)| \bigg( \prod_{\alpha \in R^+} |e^{2 \pi i \alpha (X)}-1|^2 \bigg) \, \mathrm{d} m(X) \]
where $a=\min_{\lambda \in \Gamma_{\mathrm{root}}^*} |\lambda|/2=\min_{\alpha \in R} |\alpha|/2$ and $\rho^+ = \sum_{\alpha \in R^+} \alpha /2$, and let
\[ \phi (t) = \inf_{0<c<2(\sqrt{n^2+n}-n)} \sqrt{\frac{n}{1-c-c^2/(4n)}} \cdot \frac{g \left( \frac{c}{nt} \right)}{\frac{c}{nt}} . \]
Then for any real number $M \ge |2\rho^+|+a$,
\[ W_g (\nu_1, \nu_2) \le \psi (M) + \phi (M) \left( \sum_{\substack{\pi \in \widehat{G} \\ 0<|\lambda_{\pi}|<M}} \frac{d_{\pi}}{\kappa_{\pi}} \| \widehat{\nu_1}(\pi) - \widehat{\nu_2}(\pi) \|_{\mathrm{HS}}^2 \right)^{1/2} . \]
\end{thm}
\noindent Observe that $\psi (M) \ll g(1/M)$ and $\phi (M) \ll Mg(1/M)$ with implied constants depending only on $G$. Theorem \ref{WpBerryEsseen} thus follows from Theorem \ref{WgBerryEsseen} with $g(t)=t^p$, $0<p\le 1$.

\subsection{Applications}\label{applications}

\subsubsection{Spectral gaps and random walks}

Given Borel probability measures $\nu_1$ and $\nu_2$ on $G$, let $\nu_1 * \nu_2$ denote their convolution, and let $\nu_1^*(B)=\nu_1(B^{-1})$, $B \subseteq G$ Borel. If $\zeta_1$ and $\zeta_2$ are independent $G$-valued random variables with distribution $\nu_1$ and $\nu_2$, then $\nu_1 * \nu_2$ (resp.\ $\nu_1^*$) is the distribution of $\zeta_1 \zeta_2$ (resp.\ $\zeta_1^{-1}$).

Let $L_0^2(G,\mu_G)$ be the orthogonal complement of the space of constant functions in $L^2(G,\mu_G)$; that is, the set of all $f \in L^2(G,\mu_G)$ with $\int_G f \, \mathrm{d}\mu_G =0$. Given a Borel probability measure $\nu$ on $G$, let $T_{\nu}: L_0^2(G,\mu_G) \to L_0^2(G,\mu_G)$,
\[ (T_{\nu}f)(x)= \int_G f(xy) \, \mathrm{d}\nu (y) \]
be its associated Markov operator. Observe that $T_{\nu_1*\nu_2}=T_{\nu_1} T_{\nu_2}$ and $T_{\nu^*}=T_{\nu}^*$; in particular, $T_{\nu}$ is self-adjoint (resp.\ normal) if and only if $\nu=\nu^*$ (resp.\ $\nu * \nu^* = \nu^* * \nu$).

We start with a trivial estimate for $W_p (\nu, \mu_G)$ in terms of $T_{\nu}$. It is not difficult to see that
\[ q_{\nu} := \| T_{\nu} \|_{\mathrm{op}} = \sup_{\substack{\pi \in \widehat{G} \\ \pi \neq \pi_0}} \| \widehat{\nu} (\pi ) \|_{\mathrm{op}}, \]
where $\pi_0 \in \widehat{G}$ denotes the trivial representation, and $\| \cdot \|_{\mathrm{op}}$ is the operator norm. Let $f \in \mathcal{F}_p$ with $\int_G f \, \mathrm{d}\mu_G=0$ be arbitrary, and note that $T_{\nu} f \in \mathcal{F}_p$. Since $|T_{\nu} f| \ge |(T_{\nu} f)(e)|/2$ on the ball centered at $e$ with radius $r=(|(T_{\nu} f)(e)|/2)^{1/p}$, we have
\[ \begin{split} \| T_{\nu} f \|_2^2 &\ge \left( \frac{(T_{\nu}f)(e)}{2} \right)^2 \mu_G \left( B(e,r) \right) \gg |(T_{\nu}f)(e)|^{2+n/p}, \\ \| T_{\nu}f \|_2^2 &\le \| T_{\nu} \|_{\mathrm{op}}^2 \cdot \| f \|_2^2 \ll \| T_{\nu} \|_{\mathrm{op}}^2 . \end{split} \]
Therefore $|\int_G f \, \mathrm{d} \nu| = |(T_{\nu}f)(e)| \ll q_{\nu}^{2p/(n+2p)}$, and consequently
\begin{equation}\label{trivialqnu}
W_p (\nu, \mu_G ) \ll q_{\nu}^{2p/(n+2p)} .
\end{equation}

We now deduce a sharp improvement on the trivial estimate \eqref{trivialqnu}. Recall that $\| A \|_{\mathrm{HS}} \le \sqrt{d_{\pi}} \| A \|_{\mathrm{op}}$ for any $d_{\pi} \times d_{\pi}$ matrix $A$. With $\nu_1=\nu$ and $\nu_2=\mu_G$,
\begin{equation}\label{dimensionestimate}
\begin{split} \sum_{\substack{\pi \in \widehat{G}\\ 0<|\lambda_{\pi}|<M}} \frac{d_{\pi}}{\kappa_{\pi}} \| \widehat{\nu_1}(\pi) - \widehat{\nu_2}(\pi ) \|_{\mathrm{HS}}^2 &\le \sum_{\substack{\pi \in \widehat{G} \\ 0<|\lambda_{\pi}|<M}} \frac{d_{\pi}^2}{\kappa_{\pi}} \| \widehat{\nu}(\pi ) \|_{\mathrm{op}}^2 \\ &\ll \sum_{\substack{\pi \in \widehat{G}\\ 0<|\lambda_{\pi}|<M}} |\lambda_{\pi}|^{n-r-2} q_{\nu}^2 \\ &\ll \left\{ \begin{array}{ll} q_{\nu}^2 & \textrm{if } n=1, \\ (\log (M+2)) q_{\nu}^2 & \textrm{if } n=2, \\ M^{n-2} q_{\nu}^2 & \textrm{if } n \ge 3 . \end{array} \right. \end{split}
\end{equation}
Optimizing the value of the free parameter $M>0$, in dimension $n \ge 3$ Theorem \ref{WpBerryEsseen} thus gives that for any $0<p \le 1$,
\begin{equation}\label{improvedqnu}
W_p (\nu, \mu_G ) \ll q_{\nu}^{2p/n}
\end{equation}
with an implied constant depending only on $G$. Using Theorem \ref{WgBerryEsseen} instead, we get $W_g(\nu, \mu_G) \ll g(q_{\nu}^{2/n})$. Similar estimates can be deduced in dimensions $n=1$ and $2$. Clearly $q_{\nu} \le 1$, and $q_{\nu_1*\nu_2} \le q_{\nu_1} q_{\nu_2}$; in particular, \eqref{improvedqnu} gives the upper bound for the rate of convergence of random walks $W_p (\nu_1 * \cdots * \nu_N, \mu_G) \ll \prod_{k=1}^N q_{\nu_k}^{2p/n}$. 

We say that $\nu$ has a \textit{spectral gap}, if the spectral radius of $T_{\nu}$ is strictly less than $1$; note that this is a direct generalization of Cram\'er's condition in classical probability theory. Assuming $T_{\nu}$ is normal, having a spectral gap is equivalent to $q_{\nu}<1$; for general $T_{\nu}$, it is equivalent to $q_{\nu^{*m}}<1$ for some integer $m \ge 1$. Deciding whether a given $\nu$ has a spectral gap is a highly nontrivial problem. Generalizing results of Bourgain and Gamburd \cite{Bourgain1}, \cite{Bourgain2} on $\mathrm{SU}(2)$ and $\mathrm{SU}(d)$, Benoist and Saxc\'e considered a Borel probability measure $\nu$ on a compact, connected, simple Lie group $G$. They proved \cite[Theorem 3.1]{Benoist} that if the support of $\nu$ is not contained in any proper closed subgroup, and each element of the support (as a matrix) has algebraic entries, then $\nu$ has a spectral gap. The same authors also conjectured that the condition that the matrix entries are algebraic can be dropped.

Using \eqref{improvedqnu} (or even just \eqref{trivialqnu}), $W_p (\nu^{*k},\mu_G) \to 0$ exponentially fast as $k\to \infty$ whenever $\nu$ has a spectral gap. Corollary \ref{convergencerate} is thus basically an unconditional (i.e.\ not assuming the conjecture of Benoist and Saxc\'e), weaker form of this fact. In contrast to the (semi)simple case, $W_p(\nu^{*k},\mu_G) \to 0$ polynomially fast for certain finitely supported measures $\nu$ on the torus $\mathbb{R}^d / \mathbb{Z}^d$ \cite{Borda}.

So far we have only discussed the relationship between $W_p(\nu, \mu_G)$ and the spectral gap of $\nu$. Theorem \ref{WpBerryEsseen}, however, provides a quantitative relationship between $W_p (\nu, \mu_G)$ and the spectrum of the self-adjoint operator $T_{\nu}^* T_{\nu}$ itself. Indeed, by the Peter--Weyl theorem $L_0^2(G,\mu_G)=\oplus_{\pi \in \widehat{G}, \pi \neq \pi_0} V_{\pi}$, where $V_{\pi}$ is the vector space spanned by the entries of $\pi(x)$. Since $(T_{\nu}\pi)(x)=\pi (x) \widehat{\nu}(\pi)^*$, the action of $T_{\nu}$ on $V_{\pi}$ is determined by $\widehat{\nu}(\pi )$; in particular, $d_{\pi} \| \widehat{\nu}(\pi) \|_{\mathrm{HS}}^2$ is simply the sum of all spectrum points of $T_{\nu}^* T_{\nu}$ on $V_{\pi}$. The proof of Corollary \ref{convergencerate} is based on this quantitative relationship.
\begin{proof}[Proof of Corollary \ref{convergencerate}] Varj\'u \cite[Theorem 6]{Varju} proved that for any Borel probability measure $\vartheta$ on $G$ and any $M>0$,
\begin{equation}\label{nonuniform}
1- \max_{\substack{\pi \in \widehat{G} \\ 0<|\lambda_{\pi}| \le M}} \| \widehat{\vartheta}(\pi) \|_{\mathrm{op}} \ge c_0 \left( 1- \max_{\substack{\pi \in \widehat{G} \\ 0<|\lambda_{\pi}| \le M_0}} \| \widehat{\vartheta}(\pi) \|_{\mathrm{op}} \right) \frac{1}{\log^A (M+2)} ,
\end{equation}
where the constants $c_0, M_0>0$ and $1 \le A \le 2$ depend only on the group $G$; in fact, the exact value of $A$ was also given. Since $\nu^{*k} \to \mu_G$ weakly, we have $\widehat{\nu}(\pi)^k = \widehat{\nu^{*k}} (\pi) \to 0$ for all $\pi \neq \pi_0$, and hence the spectral radius of $\widehat{\nu}(\pi)$ is less than $1$. It follows that for any $\pi \in \widehat{G}$ with $0<|\lambda_{\pi}| \le M_0$, we have $\| \widehat{\nu} (\pi)^m \|_{\mathrm{op}}<1$ with some positive integer $m=m(G,\nu)$; in particular,
\[ b=b(G,\nu ) := c_0 \left( 1- \max_{\substack{\pi \in \widehat{G} \\ 0<|\lambda_{\pi}| \le M_0}} \| \widehat{\nu}(\pi)^m \|_{\mathrm{op}} \right) >0. \]
Applying \eqref{nonuniform} to $\vartheta=\nu^{*m}$, we get that for any positive integer $k$ and any $M>0$,
\[ \begin{split} \max_{\substack{\pi \in \widehat{G} \\ 0<|\lambda_{\pi}| \le M}} \| \widehat{\nu^{*k}} (\pi) \|_{\mathrm{op}} \le \left( \max_{\substack{\pi \in \widehat{G} \\ 0<|\lambda_{\pi}| \le M}} \| \widehat{\nu} (\pi)^m \|_{\mathrm{op}} \right)^{\lfloor k/m \rfloor} &\le \left( 1-\frac{b}{\log^A (M+2)} \right)^{(k-m)/m} \\ &\le e^{-b(k-m)/(m\log^A (M+2))} . \end{split} \]
Hence
\[ \begin{split} \sum_{\substack{\pi \in \widehat{G}\\ 0<|\lambda_{\pi}|<M}} \frac{d_{\pi}}{\kappa_{\pi}} \| \widehat{\nu^{*k}} (\pi ) \|_{\mathrm{HS}}^2 &\le \sum_{\substack{\pi \in \widehat{G}\\ 0<|\lambda_{\pi}|<M}} \frac{d_{\pi}^2}{\kappa_{\pi}} \| \widehat{\nu^{*k}} (\pi ) \|_{\mathrm{op}}^2 \\ &\ll \sum_{\substack{\pi \in \widehat{G}\\ 0<|\lambda_{\pi}|<M}} |\lambda_{\pi}|^{n-r-2} e^{-b(k-m)/(m\log^A (M+2))} \\ &\ll M^n e^{-b(k-m)/(m\log^A (M+2))}. \end{split} \]
The first factor is actually $1$, $\log (M+2)$, $M^{n-2}$ in the cases $n=1$, $n=2$, $n \ge 3$, but this will not play an important role. Theorem \ref{WpBerryEsseen} thus gives that for any $M>0$,
\[ W_1 (\nu^{*k}, \mu_G ) \ll \frac{1}{M} + M^{n/2} e^{-b(k-m)/(2m \log^A (M+2))} . \]
Choosing $\log^{A+1} M= b(k-m)/(2mn)$, we deduce
\[ W_1 (\nu^{*k}, \mu_G ) \ll k^{\frac{1}{A+1}} \exp \left( - \frac{n}{2} \left( \frac{b(k-m)}{2mn} \right)^{\frac{1}{A+1}} \right) . \]
In particular, $W_1 (\nu^{*k}, \mu_G ) \ll \exp \left( -ck^{\frac{1}{A+1}} \right)$ with any $0<c<\frac{n}{2} \cdot (\frac{b}{2mn})^{\frac{1}{A+1}}$.
\end{proof}

\begin{remark} Using Theorem \ref{WgBerryEsseen} instead of Theorem \ref{WpBerryEsseen}, we deduce the general form of the conclusion of Corollary \ref{convergencerate} as $W_g (\nu^{*k}, \mu_G) \ll g(e^{-ck^{1/3}})$.
\end{remark}

\subsubsection{Uniform distribution theory}

Next, we consider applications in uniform distribution theory. It is not difficult to see e.g.\ directly from the definition of $W_g$, that for any given nonempty finite set $A \subset G$ and any $g$ as in Section \ref{notation},
\begin{equation}\label{optimalquantization1}
\inf_{\mathrm{supp} \, \nu \subseteq A} W_g (\nu, \mu_G) = \int_G g(\mathrm{dist} \, (A,x)) \, \mathrm{d} \mu_G (x) ,
\end{equation}
where the infimum is over all probability measures $\nu$ whose support is contained in $A$, and $\mathrm{dist} \, (A, \cdot)$ denotes distance from the set $A$. Indeed, the infimum is attained when for any $a \in A$, $\nu (\{ a \})$ is the Haar measure of the Voronoi cell
\[ \{ x \in G \, : \, \mathrm{dist} \, (A,x) = \rho (a,x) \} . \]
In this case the optimal transport plan from $\nu$ to $\mu_G$ is to simply spread $\nu (\{ a \})$ evenly over the given Voronoi cell. Recall that open balls $B(x,r)$ in $G$ of radius $0<r<\mathrm{diam}\, G$ satisfy $r^n \ll \mu_G (B(x,r)) \ll r^n$. A standard ball packing argument (using e.g.\ the ``$3r$ covering lemma'' of Vitali) shows that the optimal distance from a probability measure supported on at most $N$ points to the Haar measure is
\begin{equation}\label{optimalquantization2}
g(N^{-1/n}) \ll \inf_{|\mathrm{supp} \, \nu | \le N} W_g (\nu, \mu_G ) \ll g(N^{-1/n})
\end{equation}
with implied constants depending only on $G$. In particular, \eqref{optimalquantization1} and \eqref{optimalquantization2} hold for $W_p$, $0<p\le 1$. We mention that $W_p$, $1 \le p < \infty$ also satisfies the same estimates as $W_1$. For a detailed proof in the case $1\le p<\infty$ see Kloeckner \cite{Kloeckner}; the proof for $0<p<1$ and for more general $g$ is identical. We refer to the same paper for far reaching generalizations (e.g.\ to more general measures on Riemannian manifolds).

Lubotzky, Phillips and Sarnak \cite{Sarnak1}, \cite{Sarnak2} considered the problem of finding well distributed finite sets in $\mathrm{SO}(3)$, and consequently, on the sphere $S^2$. For any $N$ such that $2N-1$ is a prime congruent to $1$ modulo $4$, they constructed a symmetric set $\{ a_1, a_2, \dots, a_{2N} \} \subset \mathrm{SO}(3)$ for which the probability measure $\nu_N= (2N)^{-1} \sum_{k=1}^{2N} \delta_{a_k}$ satisfies $q_{\nu_N} = \sqrt{2N-1}/N$; this spectral gap is in fact optimal among all symmetric sets of size $2N$. Since $\mathrm{SO}(3)$ has dimension $n=3$, \eqref{improvedqnu} yields that for any $0<p \le 1$,
\[ W_p (\nu_N, \mu_{\mathrm{SO}(3)}) \ll N^{-p/3} \]
with a universal implied constant; by \eqref{optimalquantization2}, this is optimal. Note that the trivial estimate \eqref{trivialqnu} only yields $W_p (\nu_N, \mu_{\mathrm{SO}(3)}) \ll N^{-p/(3+2p)}$. More generally, we have $W_g (\nu_N, \mu_{\mathrm{SO}(3)}) \ll g(N^{-1/3})$.

Clozel \cite{Clozel} proved a similar optimal (up to a constant factor) spectral gap estimate in terms of the size of a finite set in $\mathrm{U}(d)$. Less precise estimates on more general compact homogeneous spaces were obtained by Oh \cite{Oh}.

\subsubsection{Empirical measures}\label{empirical}

Finally, we address the sharpness of Theorems \ref{WpBerryEsseen} and \ref{WgBerryEsseen}; we do so by deducing a simple estimate on the mean rate of convergence of empirical measures. Let $\nu$ be an arbitrary Borel probability measure on $G$, and let $\zeta_1, \zeta_2, \dots, \zeta_N$ be independent, identically distributed $G$-valued random variables with distribution $\nu$. The probability measure $\overline{\nu}_N:=N^{-1} \sum_{k=1}^N \delta_{\zeta_k}$ is called the corresponding empirical measure. Theorem \ref{WpBerryEsseen} gives an estimate for $W_p (\overline{\nu}_N, \nu)$ --- a random variable! --- as follows. Let $E_{\pi} = \mathbb{E} \pi (\zeta_1 ) = \widehat{\nu}(\pi )^*$. With $\nu_1 = \overline{\nu}_N$ and $\nu_2 = \nu$ we then have
\[ \widehat{\nu_1} (\pi ) - \widehat{\nu_2} (\pi ) = \frac{1}{N} \sum_{k=1}^N \left( \pi (\zeta_k ) -E_{\pi} \right)^* , \]
and by independence, the ``variance'' satisfies
\[ \mathbb{E} \, \| \widehat{\nu_1} (\pi ) - \widehat{\nu_2}(\pi ) \|_{\mathrm{HS}}^2 = \frac{1}{N^2} \sum_{k=1}^N \E \, \mathrm{tr} \left( \pi (\zeta_k)^* \pi (\zeta_k) - E_{\pi}^* E_{\pi} \right) \le \frac{d_{\pi}}{N} . \]
In the last step we used that $\pi (x)$ is unitary. Following the steps in \eqref{dimensionestimate}, in dimension $n \ge 3$ Theorem \ref{WpBerryEsseen} gives that for any $0<p \le 1$ and any $M>0$,
\[ \mathbb{E} W_p (\overline{\nu}_N, \nu ) \le \sqrt{\E W_p (\overline{\nu}_N, \nu )^2} \ll \frac{1}{M^p} + M^{1-p} \sqrt{\frac{M^{n-2}}{N}} . \]
Optimizing the value of $M>0$, we finally obtain that in dimension $n \ge 3$, for any $0<p \le 1$,
\begin{equation}\label{empiricalrate}
\mathbb{E} W_p (\overline{\nu}_N, \nu ) \ll N^{-p/n}
\end{equation}
with an implied constant depending only on $G$; more generally, $\mathbb{E} W_g (\overline{\nu}_N, \nu) \ll g(N^{-1/n})$. These are sharp by \eqref{optimalquantization2}; in particular, Theorems \ref{WpBerryEsseen} and \ref{WgBerryEsseen} are also sharp up to a constant factor depending on $G$. Note that the only compact, connected Lie groups in dimension $n=1$ and $n=2$ are $\mathbb{R}/\mathbb{Z}$ and $\mathbb{R}^2/\mathbb{Z}^2$, and the sharpness of Theorem \ref{WpBerryEsseen} on these groups follows from results in \cite{Borda}.

The rate of convergence of empirical measures in $W_p$, $p \ge 1$ on more general metric spaces was studied by Bach and Weed \cite{Bach}, and by Boissard and Le Gouic \cite{Boissard}. Instead of Fourier methods, they used a sequence of partitions of the metric space, each refining its predecessor to construct transport plans. It follows e.g.\ from \cite[Corollary 1.2]{Boissard} that our estimate \eqref{empiricalrate} can be improved to $\E W_p (\overline{\nu}_N, \nu) \ll N^{-1/n}$ for all $p \ge 1$. We refer to \cite{Bach} for improvements for measures $\nu$ supported on sets of lower dimension than the ambient space.

\section{Proof of Theorem \ref{WgBerryEsseen}}\label{proofs}

The proof of Berry--Esseen type inequalities are usually based on smoothing with an approximate identity whose Fourier transform has bounded support. For instance, in the proof of Theorem \ref{classicalBerryEsseen} this Fourier transform is the ``rooftop function'' $\max \{ 1-|t|/T, 0 \}$, supported on $[-T,T]$. The proof of Theorem \ref{NiederreiterBerryEsseen} uses the discrete version $\prod_{k=1}^d \max\{1-|m_k|/(M+1), 0 \}$, supported on $[-M,M]^d$; in the setting of the torus this is known as the Fej\'er kernel.

Our proof of Theorem \ref{WgBerryEsseen} follows the same idea. We will choose a kernel $K_M: G \to \mathbb{C}$ whose Fourier transform satisfies $\widehat{K_M}(\pi)=0$ whenever $|\lambda_{\pi}| \ge M$. Clearly,
\begin{equation}\label{basicestimate}
\left| \int_G f \, \mathrm{d}\nu_1 - \int_G f \, \mathrm{d}\nu_2 \right| \le 2 \| f-f*K_M \|_{\infty} + \left| \int_G f*K_M \, \mathrm{d}\nu_1 - \int_G f*K_M \, \mathrm{d}\nu_2 \right| ,
\end{equation}
where $f*K_M$ denotes convolution. Our goal is to find an upper estimate of the right hand side which is uniform in $f \in \mathcal{F}_g$; by Kantorovich duality, the same upper estimate will hold for $W_g (\nu_1, \nu_2)$. A possible choice for $K_M$ could be a Fej\'er-like kernel
\[ \frac{1}{|B_M|} \bigg| \sum_{\substack{\pi \in \widehat{G} \\ \lambda_{\pi} \in B_M}} \chi_{\pi} \bigg|^2 \]
with some set $B_M \subseteq \{ \lambda \in \Gamma^* \cap C^+ \, : \, |\lambda|<M/2 \}$. For convergence properties of such Fej\'er kernels on compact Lie groups we refer to \cite{Brandolini} and \cite{Travaglini}. We mention that using these kernels it is possible to deduce upper bounds to $W_p$, $0<p<1$ sharp up to a constant factor depending on $p$, but in the case $p=1$ we necessarily lose a logarithmic factor in $M$; the reason is that the Fej\'er kernel does not approximate Lipschitz functions optimally in the supremum norm. Fixing this shortcoming is easy on the torus; we simply need to use the normalized square of the Fej\'er kernel instead. By Jackson's theorem we then have the optimal rate of approximation $\| f-f*K_M \|_{\infty} \ll g(1/M)$, and a sharp Berry--Esseen smoothing inequality on the torus follows \cite{Borda}. Similar modifications of the Fej\'er kernel are known to yield Jackson type theorems on certain classical groups, see Gong \cite{Gong}; however, this approach seems not to have been worked out in full generality. An elegant proof of Jackson's theorem on an arbitrary compact, connected Lie group was nevertheless found by Cartwright and Kucharski \cite{Cartwright}, and in this paper we will use their kernel.

For the sake of completeness, we include the construction of the kernel in Section \ref{kernelsection}; we carry out the smoothing procedure in Section \ref{smoothingsection}; finally, prove a decay estimate for the Fourier transform of $f$ and finish the proof of Theorem \ref{WgBerryEsseen} in Section \ref{fourierdecay}.

\subsection{Construction of the kernel}\label{kernelsection}

Recall that the Weyl integral formula \cite[p.\ 338]{Bourbaki3} states that for any central function $\varphi \in L^1(G,\mu_G)$ we have
\[ \int_G \varphi \, \mathrm{d}\mu_G = \frac{1}{|W(G,T)|}\int_T \varphi \cdot \delta_G \, \mathrm{d}\mu_T , \]
where the function $\delta_G : T \to \mathbb{R}$ is defined as
\[ \delta_G (\exp (2 \pi X))=\prod_{\alpha \in R} (e^{2 \pi i \alpha (X)}-1) = \prod_{\alpha \in R^+} \left| e^{2 \pi i \alpha (X)} -1 \right|^2 , \qquad X \in \mathfrak{t}. \]
In particular, $\delta_G \ge 0$ and $\int_T \delta_G \, \mathrm{d}\mu_T = |W(G,T)|$. Expanding the product in its definition, $\delta_G$ is thus a $W(G,T)$-invariant trigonometric polynomial on $T$ of the form
\begin{equation}\label{deltagtrigpolynomial}
\delta_G (\exp (2 \pi X)) = \sum_{\lambda \in \Gamma_{\mathrm{root}}^*} c_{\lambda} e^{2 \pi i \lambda (X)} , \qquad X \in \mathfrak{t}.
\end{equation}
The constant term is $c_0=|W(G,T)|$, and all coefficients satisfy $|c_{\lambda}| \le |W(G,T)|$. Observe also that for any $\varphi \in L^1 (T, \mu_T)$,
\begin{equation}\label{exponentialtransformation}
\int_T \varphi (t) \, \mathrm{d}\mu_T (t) = \int_{\mathfrak{t}/\Gamma} \varphi (\exp (2 \pi X)) \, \mathrm{d}\mu_{\mathfrak{t}/\Gamma}(X),
\end{equation}
where $\mu_{\mathfrak{t}/\Gamma}$ is the normalized Haar measure on $\mathfrak{t}/\Gamma$. Note that $\mu_{\mathfrak{t}/\Gamma} = m/\mathrm{Vol} (\mathfrak{t}/\Gamma)$, where $\mathrm{Vol} (\mathfrak{t}/\Gamma)$ is the Lebesgue measure of the fundamental domain of $\Gamma$.

Following \cite{Cartwright} with minor modifications, we now construct a kernel $K_M$. Let $\eta (X)$ and $F (X)$ be as in Section \ref{BerryEsseensection}, and recall the notation $a=\min_{\lambda \in \Gamma_{\mathrm{root}}^*} |\lambda|/2$. Let $M \ge |2 \rho^+| +a$ be arbitrary, and set $M_0=\lfloor M/(|2\rho^+|+a) \rfloor$. Define $P: T \to \mathbb{C}$ as
\[ P (\exp (2 \pi X)) = \mathrm{Vol} (\mathfrak{t}/\Gamma ) (aM_0)^r \sum_{Y \in \Gamma} F(aM_0(X+Y)) , \qquad X \in \mathfrak{t} . \]
Note that $P$ is well-defined, smooth and $W(G,T)$-invariant. By \eqref{exponentialtransformation}, its Fourier coefficient with respect to $\lambda \in \Gamma^*$ (i.e.\ the character $\exp (2 \pi X) \mapsto e^{2 \pi i \lambda (X)}$ on $T$) is
\[ \begin{split} \widehat{P}(\lambda ) &= \int_{\mathfrak{t}/\Gamma} P (\exp (2 \pi X)) e^{-2 \pi i \lambda (X)} \, \mathrm{d} \mu_{\mathfrak{t}/\Gamma} (X) \\ &= \int_{\mathfrak{t}} (aM_0)^r F(aM_0 X) e^{-2 \pi i \lambda (X)} \, \mathrm{d} m (X) \\ &= \eta (\lambda^* / (aM_0)) , \end{split} \]
where $\lambda^*$ is the unique element in $\mathfrak{t}$ with $\lambda (X) = (\lambda^*, X)$. By the construction of $\eta$, $\widehat{P}(\lambda)=0$ whenever $|\lambda| \ge aM_0$; consequently, $P$ is a $W(G,T)$-invariant trigonometric polynomial on $T$ with degree $<aM_0$. Observe also that
\[ \frac{\delta_G (t^{M_0})}{\delta_G (t)} = \prod_{\alpha \in R} \frac{e^{2 \pi i M_0 \alpha (X)}-1}{e^{2 \pi i \alpha (X)}-1} \qquad (t=\exp (2 \pi X)) \]
is a $W(G,T)$-invariant trigonometric polynomial on $T$ of degree $\le |2 \rho^+| (M_0-1)$. Hence $P(t) \delta_G (t^{M_0})/\delta_G (t)$ is a $W(G,T)$-invariant trigonometric polynomial on $T$ of degree $<aM_0+|2 \rho^+| (M_0-1)<M$. It follows (see e.g.\ \cite[Lemma 1]{Cartwright}), that there exists a central trigonometric polynomial $K_M$ on $G$ of degree $<M$ --- that is, a function $K_M: G \to \mathbb{C}$ of the form $K_M=\sum_{\pi \in \widehat{G}, |\lambda_{\pi}|<M} a_{\pi} \chi_{\pi}$ --- such that
\[ K_M (t) = P(t) \frac{\delta_G (t^{M_0})}{\delta_G (t)} \qquad \textrm{for all } t \in T. \]

\subsection{The smoothing procedure}\label{smoothingsection}

First, we estimate the coefficients in $K_M=\sum_{\pi \in \widehat{G}, |\lambda_{\pi}|<M} a_{\pi} \chi_{\pi}$. Let $W(\pi)$ denote the set of weights of a representation $\pi \in \widehat{G}$. Since there exists a unitary matrix $U$ such that $U \pi(\exp (2 \pi X)) U^* = \mathrm{diag} \, (e^{2 \pi i \mu (X)} \, : \, \mu \in W(\pi))$, $X \in \mathfrak{t}$, we have
\[ \chi_{\pi} (\exp (2 \pi X)) = \sum_{\mu \in W(\pi )} e^{2 \pi i \mu (X)}, \qquad X \in \mathfrak{t}. \]
Therefore
\[ \begin{split} a_{\pi} &= \int_G K_M \overline{\chi_{\pi}} \, \mathrm{d} \mu_G \\ &= \frac{1}{|W(G,T)|} \int_{\mathfrak{t}/\Gamma} P(\exp (2 \pi X)) \delta_G (\exp (2 \pi M_0 X)) \overline{\chi_{\pi} (\exp (2 \pi X))} \, \mathrm{d} \mu_{\mathfrak{t}/\Gamma} (X) \\ &= \int_{\mathfrak{t}} (aM_0)^r F(aM_0 X) \bigg( \sum_{\lambda \in \Gamma_{\mathrm{root}}^*} \frac{c_{\lambda}}{|W(G,T)|} e^{2 \pi i M_0 \lambda (X)} \bigg) \bigg( \sum_{\mu \in W(\pi )} e^{-2 \pi i \mu (X)} \bigg) \, \mathrm{d}m(X) \\ &= \sum_{\mu \in W(\pi )} \sum_{\lambda \in \Gamma_{\mathrm{root}}^*} \frac{c_{\lambda}}{|W(G,T)|} \eta (-\lambda^* /a+\mu^* /(aM_0)) ,  \end{split} \]
where $\lambda^* \in \mathfrak{t}$ is such that $\lambda (X) = (\lambda^*, X)$. By the construction of $\eta$, for any given $\mu \in W(\pi)$ there is at most one $\lambda \in \Gamma_{\mathrm{root}}^*$ for which $\eta (-\lambda^* /a+\mu^* /(aM_0)) \neq 0$. Recalling that $|c_{\lambda}| \le |W(G,T)|$ and $0 \le \eta \le 1$, it follows that the coefficients of $K_M$ satisfy $|a_{\pi}| \le d_{\pi}$. For the trivial character the only nonvanishing term is $\lambda =0$; in particular, $\int_G K_M \, \mathrm{d} \mu_G =(c_0/|W(G,T)|) \eta (0)=1$.

\begin{remark} In fact, the only nonvanishing term is $\lambda =0$ whenever $|\lambda_{\pi}| \le aM_0$. Assuming in addition, that $\eta (X)=1$ for all $|X| \le 1/2$, we thus have $a_{\pi}=d_{\pi}$ whenever $|\lambda_{\pi}| \le aM_0/2$. In other words, $f*K_M=f$ for any central trigonometric polynomial of the form $f=\sum_{\pi \in \widehat{G}, |\lambda_{\pi}|\le aM_0/2} b_{\pi} \chi_{\pi}$. Thus $K_M$ is an analogue of the de la Vall\'ee Poussin kernel, although its construction is not based on the Fej\'er kernel.
\end{remark}

\begin{prop}\label{smoothingprop} For any $f \in \mathcal{F}_g$ and any real $M \ge |2\rho^+|+a$,
\[ \left| \int_G f \, \mathrm{d} \nu_1 - \int_G f \, \mathrm{d} \nu_2 \right| \le \psi (M) + \sum_{\substack{\pi \in \widehat{G} \\ |\lambda_{\pi}|<M}} d_{\pi} \| \widehat{f} (\pi ) \|_{\mathrm{HS}} \cdot \| \widehat{\nu_1} (\pi ) - \widehat{\nu_2} (\pi ) \|_{\mathrm{HS}} , \]
where $\psi$ is as in Theorem \ref{WgBerryEsseen}.
\end{prop}

\begin{proof} Recall \eqref{basicestimate}. We first show that $2 \| f-f*K_M \|_{\infty} \le \psi (M)$; with somewhat weaker constant factors this was proved in \cite{Cartwright}. Since the geodesic metric $\rho$ is translation invariant both from the left and from the right, from the Weyl integral formula, \eqref{exponentialtransformation} and \eqref{deltagtrigpolynomial} we deduce
\[ \begin{split} \| f&-f*K_M \|_{\infty} \\ &= \sup_{x \in G} \left| \int_G \left( f(x)-f(xy^{-1}) \right) K_M(y) \, \mathrm{d}\mu_G (y) \right| \\ &\le \int_G g(\rho (e,y)) |K_M(y)| \, \mathrm{d}\mu_G (y) \\ &= \frac{1}{|W(G,T)|} \int_{\mathfrak{t}/\Gamma} g(\rho (e,\exp (2 \pi X))) |P(\exp (2 \pi X))| \delta_G( \exp (2 \pi M_0 X)) \, \mathrm{d} \mu_{\mathfrak{t}/\Gamma} (X) \\ &\le \frac{1}{|W(G,T)|} \int_{\mathfrak{t}} g(2 \pi |X|) (aM_0)^r |F(aM_0X)| \bigg( \prod_{\alpha \in R^+} |e^{2 \pi i M_0 \alpha (X)}-1|^2 \bigg) \, \mathrm{d} m(X) \\ &= \psi (M)/2. \end{split} \]
In the penultimate step we used that $g(\rho (e, \exp (2 \pi X))) \le g(2 \pi |X|)$, as the exponential map is a geodesic of unit speed. Finally, using $|a_{\pi}| \le d_{\pi}$ we get
\[ \begin{split} \left| \int_G f*K_M \, \mathrm{d}\nu_1 - \int_G f*K_M \, \mathrm{d}\nu_2 \right| &\le \sum_{\substack{\pi \in \widehat{G}\\ |\lambda_{\pi}|< M}} d_{\pi} \left| \int_G f*\chi_{\pi} \, \mathrm{d}\nu_1 - \int_G f*\chi_{\pi} \, \mathrm{d}\nu_2 \right| \\ &= \sum_{\substack{\pi \in \widehat{G}\\ |\lambda_{\pi}|<M}} d_{\pi} \left| \mathrm{tr} \left( \widehat{f}(\pi)^* \left( \widehat{\nu_1}(\pi) - \widehat{\nu_2}(\pi) \right) \right) \right| \\ &\le \sum_{\substack{\pi \in \widehat{G}\\ |\lambda_{\pi}|<M}} d_{\pi} \| \widehat{f}(\pi) \|_{\mathrm{HS}} \cdot \left\| \widehat{\nu_1}(\pi) - \widehat{\nu_2}(\pi) \right\|_{\mathrm{HS}} , \end{split} \]
which proves the claim.
\end{proof}

\subsection{Decay of the Fourier transform}\label{fourierdecay}

We prove a decay estimate for the Fourier transform in somewhat greater generality than what we need, and then finish the proof of Theorem \ref{WgBerryEsseen}.
\begin{prop}\label{fourierdecayprop}
Assume that $f \in L^1 (G,\mu_G)$ satisfies
\[ \left( \int_G |f(xh)-f(x)|^2 \, \mathrm{d}\mu_G (x) \right)^{1/2} \le g(\rho (h,e)) \]
for all $h \in G$ with some nondecreasing function $g: [0,\infty ) \to [0, \infty )$. Then for any real number $M>0$,
\[ \sum_{\substack{\pi \in \widehat{G} \\ |\lambda_{\pi}| \le M}} d_{\pi} \kappa_{\pi} \| \widehat{f} (\pi) \|_{\mathrm{HS}}^2 \le \inf_{0<c<2(\sqrt{n^2+n}-n)} \frac{n}{1-c-c^2/(4n)} \cdot \frac{g \left( \frac{c}{nM} \right)^2}{\left( \frac{c}{nM} \right)^2} . \]
\end{prop}
\noindent If $g(t)=t^p$ with some $0<p\le 1$, we can choose e.g.\ $c=(\sqrt{17}-3)/2$ (this is optimal in the worst case $p \to 0$, $n=1$) yielding
\begin{equation}\label{fourierdecayp}
\sum_{\substack{\pi \in \widehat{G} \\ |\lambda_{\pi}| \le M}} d_{\pi} \kappa_{\pi} \| \widehat{f} (\pi) \|_{\mathrm{HS}}^2 \le 9 n^{3-2p} M^{2-2p}.
\end{equation}
In the special case $p=1$ the factor $9$ can be removed, since the optimal choice is then to let $c \to 0$ (and $M \to \infty$). An estimate similar to \eqref{fourierdecayp} has recently been proved by Daher, Delgado and Ruzhansky \cite{Daher}, with an unspecified implied constant in the place of $9n^{3-2p}$. Our main improvement is that this implied constant does not depend on $f$; a crucial feature in the study of the $p$-Wasserstein metric.

\begin{proof}[Proof of Proposition \ref{fourierdecayprop}] We follow ideas in \cite{Daher}. For the sake of simplicity, we shall think about $\pi \in \widehat{G}$ as a $d_{\pi} \times d_{\pi}$ unitary matrix-valued function on $G$. For any matrix $A \in \mathbb{C}^{d_{\pi} \times d_{\pi}}$ let $\| A \|_{\mathrm{op}}=\sup \{ |Av| \, : \, v \in \mathbb{C}^{d_{\pi}}, |v|=1 \}$ and $\| A \|_{\mathrm{HS}}=\sqrt{\mathrm{tr}\, (A^*A)}$ denote the operator norm and the Hilbert--Schmidt norm, respectively. The operator norm is submultiplicative; further, for all $A,B \in \mathbb{C}^{d_{\pi} \times d_{\pi}}$ we have $\| AB \|_{\mathrm{HS}} \le \| A \|_{\mathrm{op}} \cdot \| B \|_{\mathrm{HS}}$, and the Cauchy--Schwarz inequality $|\mathrm{tr}\, (A^*B)| \le \| A \|_{\mathrm{HS}} \cdot \| B \|_{\mathrm{HS}}$.

One readily verifies the identity
\[ \left( \pi (h)-I_{d_{\pi}} \right) \widehat{f}(\pi ) = \int_G \left( f(xh)-f(x) \right) \pi (x)^* \, \mathrm{d}\mu_G (x) , \]
where $I_{d_{\pi}}$ denotes the $d_{\pi} \times d_{\pi}$ identity matrix. By the Parseval formula and the assumption on $f$, for any $h \in G$ we have
\[ \begin{split} \sum_{\pi \in \widehat{G}} d_{\pi} \mathrm{tr} \left( \left( \pi (h)-I_{d_{\pi}} \right)^* \left( \pi (h)-I_{d_{\pi}} \right) \widehat{f}(\pi ) \widehat{f}(\pi )^* \right) &= \int_G |f(xh)-f(x)|^2 \, \mathrm{d} \mu_G (x) \\ &\le g(\rho (h,e))^2 . \end{split} \]
Since the exponential map is a geodesic, we have $\rho (\exp (uX),e) \le |uX|$ for all $X \in \mathfrak{g}$ and $u \in \mathbb{R}$. For any $h=\exp (uX)$ the previous estimate thus yields
\begin{equation}\label{upperbound}
\sum_{\pi \in \widehat{G}} d_{\pi} \mathrm{tr} \left( \left( \pi (h)-I_{d_{\pi}} \right)^* \left( \pi (h)-I_{d_{\pi}} \right) \widehat{f}(\pi ) \widehat{f}(\pi )^* \right)\le g(|uX|)^2.
\end{equation}

Next, we wish to find a lower estimate. For any $X \in \mathfrak{g}$ let
\[ \mathrm{d} \pi (X) = \frac{\mathrm{d}}{\mathrm{d}u} \pi (\exp (uX)) \mid_{u=0} \in \mathbb{C}^{d_{\pi} \times d_{\pi}} \]
denote the derived representation of $\pi$.
\begin{lem}[Taylor expansion of degree 1]\label{taylorlemma} For any $X \in \mathfrak{g}$ and any $u \in \mathbb{R}$,
\[ \left\| \pi (\exp(uX)) - I_{d_{\pi}} - u \cdot \mathrm{d}\pi (X) \right\|_{\mathrm{op}} \le \frac{u^2}{2} \| \mathrm{d} \pi (X) \|_{\mathrm{op}}^2 . \]
\end{lem}

\begin{proof}[Proof of Lemma \ref{taylorlemma}]
We simply apply the usual Taylor formula to the matrix-valued function $F(u)=\pi (\exp (uX))$. Since $\pi$ is a homomorphism, we have $F'(u)=\pi (\exp (uX)) \mathrm{d} \pi (X)$. First, note that for any $u \in \mathbb{R}$,
\[ \begin{split} \left\| \pi (\exp (uX)) - I_{d_{\pi}} \right\|_{\mathrm{op}} &= \left\| \int_0^u \pi (\exp (yX)) \mathrm{d} \pi (X) \, \mathrm{d}y \right\|_{\mathrm{op}} \\ &\le \int_0^{|u|} \| \pi (\exp (yX)) \|_{\mathrm{op}} \cdot \| \mathrm{d} \pi (X) \|_{\mathrm{op}} \, \mathrm{d} y \\ &= |u| \cdot \| \mathrm{d} \pi (X) \|_{\mathrm{op}} . \end{split} \]
We used the fact that $\pi (\exp (yX))$ is a unitary matrix and thus has operator norm $1$. Therefore
\[ \begin{split} \left\| \pi (\exp (uX)) - I_{d_{\pi}} - u \cdot \mathrm{d} \pi (X) \right\|_{\mathrm{op}} &= \left\| \int_0^u \left( \pi (\exp (yX)) - I_{d_{\pi}} \right) \mathrm{d} \pi (X) \, \mathrm{d}y \right\|_{\mathrm{op}} \\ &\le \int_0^{|u|} \left\| \pi (\exp (yX)) - I_{d_{\pi}}  \right\|_{\mathrm{op}} \cdot \| \mathrm{d} \pi (X) \|_{\mathrm{op}} \, \mathrm{d} y \\ &\le \int_0^{|u|} |y| \cdot \| \mathrm{d} \pi (X) \|_{\mathrm{op}}^2 \, \mathrm{d}y \\ &= \frac{u^2}{2} \| \mathrm{d} \pi (X) \|_{\mathrm{op}}^2 .  \end{split} \]
\end{proof}

\begin{lem}[Sugiura]\label{sugiuralemma} For any $X \in \mathfrak{g}$, we have $\| \mathrm{d} \pi (X) \|_{\mathrm{op}} \le |\lambda_{\pi}| \cdot |X|$.
\end{lem}

\begin{proof}[Proof of Lemma \ref{sugiuralemma}] In \cite[Theorem 2]{Sugiura} Sugiura stated and proved the estimate $\| \mathrm{d} \pi (X) \|_{\mathrm{HS}} \le \sqrt{d_{\pi}} |\lambda_{\pi}| \cdot |X|$. His proof is based on the fact that with some $d_{\pi} \times d_{\pi}$ unitary matrix $U$, we have $U \mathrm{d} \pi (X) U^* = \mathrm{diag} \left( i \lambda (X) \, : \, \lambda \in W(\pi) \right)$, where $W(\pi)$ is the set of weights of $\pi$. Further, we have $|\lambda| \le |\lambda_{\pi}|$ for all $\lambda \in W(\pi)$. Hence Sugiura's proof in fact yields the slightly stronger claim of Lemma \ref{sugiuralemma}.
\end{proof}

\begin{lem}\label{laplacelemma} Let $X_1, \dots, X_n$ be an orthonormal base in $\mathfrak{g}$. For any $u \in \mathbb{R}$, the points $h_k=\exp (uX_k)$ satisfy
\[ \sum_{k=1}^n \left( \pi (h_k)-I_{d_{\pi}} \right)^* \left( \pi (h_k)-I_{d_{\pi}} \right) = u^2 \kappa_{\pi} I_{d_{\pi}} + E \]
with some $E\in \mathbb{C}^{d_{\pi} \times d_{\pi}}$, $\| E \|_{\mathrm{op}} \le n |u|^3 |\lambda_{\pi}|^3 + n (u^4/4) |\lambda_{\pi}|^4$.
\end{lem}

\begin{proof}[Proof of Lemma \ref{laplacelemma}] By Lemma \ref{taylorlemma} we can write
\[ \pi (h_k)-I_{d_{\pi}} = u \cdot \mathrm{d} \pi (X_k) + E_k \]
with some error matrix $E_k$ satisfying $\| E_k \|_{\mathrm{op}} \le (u^2/2) \| \mathrm{d} \pi (X_k) \|_{\mathrm{op}}^2$.
Therefore
\[ \sum_{k=1}^n \left( \pi (h_k)-I_{d_{\pi}} \right)^* \left( \pi (h_k)-I_{d_{\pi}} \right) = u^2 \sum_{k=1}^n \mathrm{d} \pi (X_k)^* \mathrm{d} \pi (X_k) +E \]
where
\[ \begin{split} \| E \|_{\mathrm{op}} &= \left\| \sum_{k=1}^n \left(u \cdot \mathrm{d} \pi (X_k)^* E_k + E_k^* u \cdot \mathrm{d} \pi (X_k) + E_k^*E_k  \right) \right\|_{\mathrm{op}} \\ &\le \sum_{k=1}^n \left( 2 |u| \cdot \| \mathrm{d} \pi (X_k) \|_{\mathrm{op}} \cdot \frac{u^2}{2} \| \mathrm{d} \pi (X_k) \|_{\mathrm{op}}^2 + \frac{u^4}{4} \| \mathrm{d} \pi (X_k) \|_{\mathrm{op}}^4 \right) .  \end{split} \]
By Lemma \ref{sugiuralemma}, the previous estimate yields $\| E \|_{\mathrm{op}} \le n |u|^3 |\lambda_{\pi}|^3 + n (u^4/4) |\lambda_{\pi}|^4$. On the other hand, we have $\mathrm{d} \pi (X)^* = -\mathrm{d} \pi (X)$, and by the definition of the Laplace--Beltrami operator,
\[ \sum_{k=1}^n \mathrm{d} \pi (X_k)^* \mathrm{d} \pi (X_k) = - \sum_{k=1}^n \mathrm{d} \pi (X_k) \mathrm{d} \pi (X_k) = -(\Delta \pi)(e) = \kappa_{\pi} I_{d_{\pi}} . \]
\end{proof}

We now finish the proof of Proposition \ref{fourierdecayprop}. Recall that $|\lambda_{\pi}|^2 \le \kappa_{\pi}$. From Lemma \ref{laplacelemma} we deduce that for any $u \in \mathbb{R}$,
\begin{equation}\label{lowerbound}
\begin{split} \mathrm{tr} \bigg( \sum_{k=1}^n \left( \pi (h_k)-I_{d_{\pi}} \right)^* &\left( \pi (h_k)-I_{d_{\pi}} \right) \widehat{f}(\pi ) \widehat{f}(\pi )^* \bigg) \\ &= \mathrm{tr} \left( u^2 \kappa_{\pi} \widehat{f}(\pi ) \widehat{f}(\pi )^*  \right) + \mathrm{tr} \left( E \widehat{f}(\pi ) \widehat{f}(\pi )^* \right) \\ &\ge u^2 \kappa_{\pi} \| \widehat{f}(\pi ) \|_{\mathrm{HS}}^2 - \| E \widehat{f}(\pi ) \|_{\mathrm{HS}} \cdot \| \widehat{f}(\pi ) \|_{\mathrm{HS}} \\ &\ge u^2 \kappa_{\pi} \| \widehat{f}(\pi ) \|_{\mathrm{HS}}^2 - \| E \|_{\mathrm{op}} \cdot \| \widehat{f}(\pi ) \|_{\mathrm{HS}}^2 \\ &\ge \| \widehat{f} (\pi ) \|_{\mathrm{HS}}^2 \left( u^2 \kappa_{\pi} - n |u|^3 |\lambda_{\pi}|^3 - n \frac{u^4}{4} |\lambda_{\pi}|^4 \right) \\ &\ge \| \widehat{f} (\pi ) \|_{\mathrm{HS}}^2 \cdot u^2 \kappa_{\pi} \left( 1 - n |u| \cdot |\lambda_{\pi}| - n \frac{u^2}{4} |\lambda_{\pi}|^2 \right) . \end{split}
\end{equation}
Let $M>0$ and $0<c<2(\sqrt{n^2+n}-n)$ be arbitrary, and choose $u=c/(nM)$. For any $|\lambda_{\pi}| \le M$ we then have
\[ 1 - n |u| \cdot |\lambda_{\pi}| - n \frac{u^2}{4} |\lambda_{\pi}|^2 \ge 1-c-\frac{c^2}{4n} >0, \]
and thus \eqref{upperbound} and \eqref{lowerbound} imply
\[ \begin{split} n g \left( \frac{c}{nM} \right)^2 &\ge \sum_{\substack{\pi \in \widehat{G} \\ |\lambda_{\pi}| \le M}} d_{\pi} \mathrm{tr} \left( \sum_{k=1}^n \left( \pi (h_k)-I_{d_{\pi}} \right)^* \left( \pi (h_k)-I_{d_{\pi}} \right) \widehat{f}(\pi ) \widehat{f}(\pi )^* \right) \\ &\ge \sum_{\substack{\pi \in \widehat{G} \\ |\lambda_{\pi}| \le M}} d_{\pi} \| \widehat{f} (\pi ) \|_{\mathrm{HS}}^2 \left( \frac{c}{nM} \right)^2 \kappa_{\pi} \left( 1-c-\frac{c^2}{4n} \right) . \end{split} \]
Since $0<c<2(\sqrt{n^2+n}-n)$ was arbitrary, the claim follows.
\end{proof}

\begin{proof}[Proof of Theorem \ref{WgBerryEsseen}] From Propositions \ref{smoothingprop} and \ref{fourierdecayprop} and the Cauchy--Schwarz inequality we get that for any $f \in \mathcal{F}_g$ and any real number $M \ge |2 \rho^+|+a$,
\[ \begin{split} \bigg| \int_G &f \, \mathrm{d}\nu_1 - \int_G f \, \mathrm{d}\nu_2 \bigg| \\ &\le \psi (M) + \left( \sum_{\substack{\pi \in \widehat{G} \\ 0<|\lambda_{\pi}|<M}} d_{\pi} \kappa_{\pi} \| \widehat{f}(\pi ) \|_{\mathrm{HS}}^2 \right)^{1/2} \left( \sum_{\substack{\pi \in \widehat{G} \\ 0<|\lambda_{\pi}|<M}} \frac{d_{\pi}}{\kappa_{\pi}} \| \widehat{\nu_1}(\pi) - \widehat{\nu_2}(\pi) \|_{\mathrm{HS}}^2 \right)^{1/2} \\ &\le \psi (M)+\phi (M) \left( \sum_{\substack{\pi \in \widehat{G} \\ 0<|\lambda_{\pi}|<M}} \frac{d_{\pi}}{\kappa_{\pi}} \| \widehat{\nu_1}(\pi) - \widehat{\nu_2}(\pi) \|_{\mathrm{HS}}^2 \right)^{1/2} . \end{split} \]
By Kantorovich duality, the same upper bound holds for $W_g (\nu_1, \nu_2)$.
\end{proof}

\subsection*{Acknowledgments}

The author is supported by the Austrian Science Fund (FWF), project Y-901. I would like to thank Daniel El-Baz, Florian Pausinger and Szil\'ard R\'ev\'esz for helpful discussions.

\end{document}